\documentclass[a4paper,11pt]{article}
\everymath{\displaystyle}
\usepackage{amssymb,amsmath,amsthm}
\usepackage{graphicx}

\usepackage{amsfonts}
\usepackage{hyperref}
\usepackage{latexsym}
\usepackage{amsthm}
\usepackage{amsmath}
\usepackage{amssymb}
\usepackage{amsopn}
\usepackage{geometry}
\usepackage{amscd}
\usepackage{mathrsfs}
\usepackage{dsfont}
\usepackage[english]{babel}
\usepackage{caption}
\usepackage{cancel}
\usepackage{ulem}

\numberwithin{equation}{section}

\font\teneufm=eufm10
\font\seveneufm=eufm7
\font\fiveeufm=eufm5
\newfam\eufmfam
\textfont\eufmfam=\teneufm
\scriptfont\eufmfam=\seveneufm
\scriptscriptfont\eufmfam=\fiveeufm

\usepackage{amsthm,amstext, amsmath,latexsym,amsbsy,amssymb,cases}
\usepackage{color}
\usepackage{authblk}
\usepackage{mathbbol}
\newtheorem{cor}{Corollary}[section]
\newtheorem{definition}[cor]{Definition}

\newtheorem{lem}[cor]{Lemma}
\newtheorem{prop}[cor]{Proposition}

\newtheorem{theorr}{Theorem}
\newtheorem{propp}[theorr]{Proposition}
\newtheorem{corr}[theorr]{Corollary}

\title{\bf Classification of the blow-up behavior for a semilinear wave equation with nonconstant degenerate coefficients}
 \usepackage[shortlabels]{enumitem}
\usepackage{hyperref}
\date{}

\usepackage{color}

\newcounter{biscompt}

\begin{document}
\author{{\large Asma Azaiez}\\

\small  \vskip-0,4cm 
\small  
University of Carthage, ISEP-BG, 2036 La Soukra, Tunisia.\\

\medskip

\noindent  {\large Hatem Zaag}\\
\small Universit\'e Sorbonne Paris Nord, LAGA, CNRS (UMR 7539),\\
\small  Villetaneuse, F-93420, France.}
\maketitle
\tableofcontents

\maketitle 
\begin{abstract}
We consider a nonlinear wave equation with nonconstant coefficients. In particular, the coefficient in front of the second order space derivative is degenerate. We give the blow-up behavior and the regularity of the blow-up set. Partial results are given at the origin, where the degeneracy occurs. 
Some nontrivial obstacles, due to the nonconstant speed of propagation, have to be surmounted.
 \end{abstract}

\section{Introduction}
We consider the following nonlinear wave equation with nonconstant coefficients in the radial case 
:
\begin{equation} \left\{
\begin{array}{l}
\displaystyle\partial^2_{t} u = a(x)\left(\partial^2_{x} u+\frac{N-1}{x}\partial_xu\right)+b (x)|u|^{p-1}u+f(u)+ g(x,t,\partial_xu, \partial_t u ), \\
 \partial_x u(x,t) \sqrt{a(x)}\rightarrow 0\mbox{ at } x= 0,\\
u(0)=u_{0} \mbox{ and }  u_{t}(0) = u_{1},
\end{array}
\right . \label{waveq}
\end{equation}
 where $u(t):x\in\mathbb{R}^+\to u(x,t) \in \mathbb{R}$, and $N$ is the dimension of the physical space.

\noindent We assume that $a\in C^1(\mathbb{R}_+^*)$ and $b\in
C^1(\mathbb{R}_+)$ satisfy the following conditions for all $x>0$,
\begin{equation} \;\; \;  \left\{
\begin{array}{l}
\;a(x)> 0, b(x)> 0,\\
 \;\int_{0}^1\frac{dx}{\sqrt{a(x)}} <+\infty, \\
  \;\;
|(N-1)\frac{\sqrt{a(x)}}{x}-\frac{1}{2}\frac{a'(x)}{\sqrt{a(x)}}-\frac{d-1}{\phi(x)}|\le M,\end{array}
\right . \label{a}
\end{equation}
for some $\epsilon_0>0$, $M>0$ and
\begin{equation}\label{3.3}d\in \mathbb{N}^*\footnote{In particular,
    at some point we will integrate with respect to the weight
    $(1-r^2)^{\frac{2}{p-1}-\frac{d-1}{2} }r^{d-1}$ which is in
    $L^1(0,1)$ if (\ref{3.3}) and (\ref{3.5}) hold.},\end{equation}
where
$\phi$ is defined by
\begin{equation}\label{phi}
\phi(x)=\int_0^x \frac{dy}{\sqrt{a(y)}}.
\end{equation}
Note that the third estimate in (\ref{a}) is useful only near the
origin. However, introducing such a change will induce a lot of pure
technical and trivial complications in the proof.
For that reason, we don't make this change, and leave it to the interested reader to check by himself that it works straightforwardly.

\medskip

The exponent $p$ is superlinear and subcritical (in relation to $d$) , in the sense that
\begin{equation}
p>1\mbox { and }p<\frac{d+3}{d-1}\mbox {  if }d\ge 2.\label{3.5}
\end{equation}

\noindent Conditions  (\ref{3.3}) and  (\ref{3.5}) will prove to be meaningful after a change of variables we perform below in (\ref{6.5}).

We assume in addition that $f$ and $g$ are $C^1$ functions, where $f:\mathbb{R}\to  \mathbb{R}$ and $g:\mathbb{R}^4\to  \mathbb{R}$ satisfy
\begin{equation} \left\{
\begin{array}{l}
|f(u)|\le M(1+|u|^q ),\mbox{  for all }u\in\mathbb{R}\mbox{  with }  (q<p, M>0),\\
|g(x,t,v,z)|\le M(1+|v|\sqrt{a(x)}+|z|),\mbox{  for all }x,t,v,z \in\mathbb{R}.
\end{array}
\right . \label{condition}
\end{equation}

\bigskip

A typical example that satisfies (\ref{a}) and which will be discussed in this paper is the following :
\begin{equation}\label{choi}
a(x)=|x|^\alpha \mbox{ with } \alpha<2.
\end{equation}
The example (\ref{choi}) shows a degeneracy at $x=0$ when $\alpha\neq 0$. Note that for $\alpha<0$, the wave speed goes to infinity and for $\alpha \in (0,2)$ it goes to zero.
For this case, conditions (\ref{a}), (\ref{3.3}) and (\ref{condition}) are fulfilled for $N\ge 2$, $|g(x,t,v,z)|\le M(1+|v||x|^{\frac{\alpha}{2}  }+|z|)$, and $d=\frac{2(N-\alpha)} { 2-\alpha}$.  \\
Note in particular that example \eqref{choi} is not
  relevant for $N=1$, since the third condition of \eqref{a} is never
  fulfilled, for any $\alpha<2$. On the contrary,  when $N=2$, example \eqref{choi} is valid for
any value $\alpha<2$, and we always have $d=2$. When $N\ge 3$, we
need to take $\alpha$ of the form $\alpha_k = \frac{2(k-N)}{k-2}$
where $k\ge 3$ is an integer, and in this case, $d= k$.

\bigskip

\noindent Initial data $(u_0, u_1)$ will be considered in the space $H_1\times H_0$ defined by
\begin{equation}\label{6.11}
H_0=\{v \in L^2_{loc}(\mathbb{R}^+)\; |\; V\in L^2_{loc, u, rad}(\mathbb{R}^+) \},
\end{equation}
\begin{equation}\label{6.3}
H_1=\{v \in L^2_{loc}(\mathbb{R}^+) \; |\; V\in H^1_{loc, u, rad}(\mathbb{R}^+) \},
\end{equation}
where 
\begin{align*}\label{}
V(X)= v(x),\;\; X=\phi(x),
\end{align*}
$\phi$ was given in (\ref{phi}),
\begin{equation*}
L^2_{loc, u, rad}(\mathbb{R}^+)=\{V \in L^2_{loc}(\mathbb{R}^+)
 \; |\;  \displaystyle\sup\limits_{r_0\ge 1} \frac{1}{r_0^{d-1}}        \int_{r_0-1}^{r_0+1} v(r)^2r^{d-1} dr <+\infty \},
\end{equation*}
and
\begin{equation*}
H^1_{loc, u, rad}(\mathbb{R}^+)=\{V \in L^2_{loc, u, rad} (\mathbb{R}^+) \; |\; \partial_r V \in L^2_{loc, u, rad}(\mathbb{R}^+) \}.
\end{equation*}

We recall the spaces $L_{loc,u}^2(\mathbb{R}^d)$ and $H_{loc,u}^1(\mathbb{R}^d)$ introduced by Antonini and Merle in \cite{AM01} by the following norms :

 $$||v||^2_{L^2_{loc,u}}=\displaystyle\sup\limits_{a\in \mathbb{ R}^d}\int_{|x-a|<1}|v(x)|^2 dx \mbox{ and }|| v||^2_{H^1_{loc,u}}=||v||^2_{L^2_{loc,u}}+||  \nabla v||^2_{L^2_{loc,u}},$$
we show in Appendix \ref{appendixA} that the spaces $L^2_{loc,u,rad}$  and $H^1_{loc,u,rad}$ are simply the radial versions of the $L^2_{loc,u}$ and  $H^1_{loc,u}$ spaces.

\medskip

Equation (\ref{waveq}) corresponds to physical situations where the wave propagates in non-homogeneous media (see for example \cite{rty}). It appears in models of traveling waves in a non-homogeneous gas with damping that changes with the position. The unknown $u$ denotes the displacement, the coefficient $a$, called the bulk modulus, accounts for changes of the temperature depending on the location.  
 
\bigskip

When $a(x)\equiv1$, this equation was considered by Hamza and Zaag in \cite{HZ12} (see also \cite{HZ2012} for some related results). 
Basically, the authors showed that the results previously proved by Merle and Zaag in \cite{MR2362418}, \cite{MR2415473}, \cite{MR2931219}  and \cite{Mz12} for the unperturbed semilinear wave equation
\begin{equation}\label{equa6.5}
\partial^2_{t}  u=
\partial^2_{x}  u+|u|^{p-1} u
\end{equation} 
do extend to the perturbed case. We also mention the work of Alexakis and Shao \cite{AS} who study the energy concentration in backward light cones near blow-up points.

\bigskip

In this paper, we want to explore the case where $a(x)\not\equiv1$. When $a$ is space dependent, we find that although the blow-up results of \cite{HZ12}  remain valid, some nontrivial obstacles have to be surmounted, in particular, at the origin where the degeneracy may occur (see for instance the typical example (\ref{choi})). Since the problem does not have a constant speed of propagation, we have to apply an appropriate transformation to obtain the desired estimates.
In fact, we remark that we can reduce to the case $a(x)\equiv 1$ thanks to the following change of variables:
\begin{equation}\label{6.5}
U(X,t)= u(x,t),\;\; X=\phi(x)\end{equation}
where $\phi$ is given in (\ref{phi}). 

Applying this transformation to (\ref{waveq}), we see that $U$ satisfies:
\begin{align*}
\partial^2_{t}U =\partial^2_{X}U+ \left((N-1)\frac{ \sqrt{a(x)}}{x} -\frac{1}{2}\frac{a'(x)}{ \sqrt{a(x)} }\right)\partial_{X}U +\beta(X) |U|^{p-1}U+f(U)+g(x,t,\frac{\partial_{X}U}{\sqrt{a(x)}},\partial_t U) 
\end{align*}
where $ \beta(X)=b(x)$ and $U(t):X\in\mathbb{R}^+\to U(X,t) \in \mathbb{R}$.

We rewrite this equation as follows
\begin{align}
\label{eq6.5}
\partial^2_{t}U =\partial^2_{X}U+\frac{d-1}{X}\partial_{X}U +\beta(X) |U|^{p-1}U+f(U)+G(X,t,\partial_{X}U,\partial_t U) 
\end{align}
with
\begin{align*}G(X,t,\partial_{X}U,\partial_t U) =g(x,t,\frac{\partial_x u}{\sqrt{a(x)}}, \partial_t u)+\left( (N-1)\frac{\sqrt{a(x)}}{x}-\frac{1}{2}  \frac{a'(x)}{\sqrt{a(x)}}-\frac{d-1}{X}\right)\partial_X U.
\end{align*}
We see from (\ref{a}) and (\ref{condition}) that we have
 \begin{align}
|f(U)|&\le M(1+|U|^q ),\mbox{  for all }U\in\mathbb{R}\mbox{  with }  (q<p, M>0),\\
|G(X,t,\partial_{X}U,\partial_t U)  |&\le M(1+|\partial_{X}U|+|\partial_t U|).\label{Hg}
\end{align} 

Note that we have,
$$\partial_X U(0,t)=0$$
thanks to the condition on the space derivative in (\ref{waveq}).
\bigskip

As for the Cauchy problem for equation (\ref{waveq}), we remark that thanks to the change of variables (\ref{6.5}), we reduce to the formalism of Hamza and Zaag in \cite{HZ}. Indeed, recalling that $(u_0,u_1)\in H_0\times H_1$, 
we derive by definition that $(U(X,0),\partial_t U(X,0))\in  L^2_{loc, u, rad}\times H^1_{loc, u, rad}$ defined in (\ref{6.11}) and (\ref{6.3}).

Therefore, as mentioned in  \cite{HZ} we use a modification of the argument by Georgiev and Todorova \cite{GT} to derive a solution $(U,\partial_t U)\in C([0,T_0),H^1_{loc,u,rad} \times L^2_{loc,u,rad})$ for some $T_0>0$. Thanks to the finite speed of propagation, we extend the definition of $U(X,t)$ to the following domain

\begin{equation*}\label{}
D_U=\{(X,t); 0\le t<T_U(X)\},
\end{equation*}
for some $1-$Lipschitz function $T_U$.

Going back to problem (\ref{waveq}), we see that we have a unique solution $(u,\partial_t u)\in C([0,T_0),H_0\times H_1)$ which is defined on a larger domain

\begin{equation*}\label{}
D_u=\{(x,t)| 0\le t<T(x)\},
\end{equation*}
where 
\begin{equation}\label{T}
T(x)=T_U(\phi (x)).
\end{equation}

Since $T'(x)=\frac{T'_U(\phi (x))}{\sqrt{a(x)}}$, it follows that $T$ is a Lipschitz function, with $\frac{1}{\sqrt{a(x)}}$ as local Lipschitz constant for $x\neq 0$. Note that $T(x)$ and $\Gamma$ will be referred as the blow-up time and the blow-up curve in the following.\\

\noindent Proceeding as in the case $a(x)\equiv1$, we introduce the following non-degeneracy condition for $\Gamma$. If we introduce for all $x \in \Bbb{R},$ $t\le T(x)$ and $\delta>0$, the generalized cone
\begin{equation}\label{2.5}
 \mathcal{C}_{x,t, \delta }=\{(\xi,\tau)\neq (x,t)\,| 0\le \tau\le t-\delta |\phi(\xi)-\phi(x)|\},
\end{equation}
then our non-degeneracy condition is the following: $x_0$ is a non-characteristic point if
\begin{equation}\label{4}
 \exists \delta = \delta (x_0) \in (0,1) \mbox{ such that } u \mbox{ is defined on }\mathcal{C}_{x_0,T(x_0), \delta_0}.
\end{equation}
If condition (\ref{4}) is not true, then we call $x_0$ a characteristic point.

\noindent We denote by $\mathcal{R}$ the set of non characteristic points and $\mathcal{S}$ the set of characteristic points.

\noindent Note that the set $ \mathcal{C}_{x,t, \delta }$ defined in (\ref{2.5}) is a cone in the variables $(X,t)$ (\ref{6.5}). In the $(x,t)$ variables, its boundary is given by the characteristics associated to the linear problem
$$\partial^2_{t}  u=a(x)\partial^2_{x}  u.$$

In order to state our results, we will use similarity variables associated to $U(X,t)$ defined in (\ref{6.5}), and which turn out to be a nonlinear version of the standard similarity variables, when related directly to $u(x,t)$:

\begin{align}\label{nchg}
w_{x_0}(y,s) =(T(x_0)-t)^\frac{2}{p-1}u(x,t), \quad  y=\frac{\phi (x)-\phi (x_0)}{T(x_0)-t}, \quad
s=-\log(T(x_0)-t).
\end{align}

Applying this transformation to (\ref{eq6.5}), we see that $w_{x_0}(y,s)$ satisfies the following equation for all $|y|<1$ and $s\ge - \log T(x_0)$ :
\begin{eqnarray} \label{equa}
\notag \partial^2_{s} w &=&(1-y^2)\partial_y^2 w_{x_0}-2 \frac{p+1}{p-1} y \partial_y w_{x_0}-2\frac{p+1}{(p-1)^2} w+ b (x_0)|w|^{p-1}w-\frac{p+3}{p-1} \partial_s w- 2 y \partial_{ys} w\\ &+& e^{-s} \frac{(d-1)}{\phi(x_0)+y e^{-s}} \partial_y w+  e^{-\frac{2ps}{p-1}}f(e^{\frac{2s}{p-1}}w)+
 (b (x)-b (x_0))|w|^{p-1}w\\ &+& e^{-\frac{2ps}{p-1}}  G(\phi(x_0)+y e^{-s},T_0-e^{-s},e^{\frac{(p+1)s}{p-1}} \partial_y w,  e^{\frac{(p+1)s}{p-1}} ( \partial_{s} w+y\partial_{y} w+ \frac{2}{p-1} w ))\notag.
 \end{eqnarray}

\bigskip

Let us introduce the solitons
\begin{eqnarray*}\label{defk}
\kappa(\hat d,y)= \kappa_0 \frac{(1-{\hat d}^2)^\frac{1}{p-1}}{(1+\hat d y)^\frac{2}{p-1}} \mbox{ with }\kappa_0=\left(\frac{2(p+1)}{(p-1)^2}\right)^\frac{1}{p-1},\,(\hat d, y) \in (-1,1)^2.
\end{eqnarray*}

We also introduce
 \begin{eqnarray}\label{4.1}
 \bar \xi_i(s)=\left(  i-\frac{k+1}{2}\right)\frac{p-1}{2} \log s+ \bar \alpha_i(p, k)
  \end{eqnarray}
  where the sequence $(\bar \alpha_i)_{i=1,...,k}$ is uniquely determined by the fact that $( \bar \xi_i(s))_{i=1,..., k}$ is an explicit solution with zero center of mass for this ODE system:
  \begin{eqnarray*}\label{}
\forall i=1,...,k,\frac{1}{c_1}  \dot{\xi_i}=e^{-\frac{2}{p-1}(\xi_i-\xi_{i-1})} -e^{-\frac{2}{p-1}(\xi_{i+1}-\xi_{i})} ,
  \end{eqnarray*} 
 where $ \xi_0(s)\equiv \xi_{k+1} (s)\equiv0$ and $c_1=c_1(p)>0$ appeared for the first time in Proposition $3.2$ page $590$ of Merle and Zaag \cite{MR2931219}.\\

\subsection{Blow-up results}
We dissociate two cases in this subsection. In fact, equation (\ref{equa}) has a different structure according to the position of $x_0$.

\subsubsection{Behavior outside the origin}
When $x_0\neq 0$, by (\ref{phi}) we have $\phi(x_0)\neq 0$, hence the term $\frac{e^{-s}}{\phi(x_0)+y e^{-s}}\partial_y w_{x_0}$ in (\ref{equa}) is a lower order term bounded by $\frac{2e^{-s}}{|\phi(x_0)|}|\partial_y w_{x_0}|$ for $s$ large and will be treated as a perturbation, as in Hamza and Zaag \cite{HZ12}.

Accordingly, we may write the second and first order space derivatives in equation (\ref{equa}) in the following divergence form:

\begin{equation*} 
(1-y^2)\partial_y^2 w_{x_0}-2 \frac{p+1}{p-1} y \partial_y w_{x_0}=\frac{1}{\rho(y)}  \partial_y (\rho(1-y^2)\partial_y w_{x_0}  )
\end{equation*}
where $\rho(y)=(1-y^2)^\frac{2}{p-1}$
exactly as in the one dimensional case of the standard semilinear wave equation (\ref{equa6.5}).

We recall that for the unperturbed case (ignoring line 2 and 3 in (\ref{equa})), the Lyapunov functional is given by
\begin{equation}\label{15}
 E(w,\partial_s w)=\int_{-1}^{1} \left( \frac{1}{2} |\partial_s w|^2+\frac{1}{2} |\partial_y w|^2 (1-y^2)+\frac{p+1}{(p-1)^2}|w|^2-\frac{b(x_0)}{p+1}|w|^{p+1}\right) \rho dy.
\end{equation}

where $(w, \partial_s w)\in H^1_\rho\times L^2_\rho$, with
 \begin{eqnarray}\label{L2}
L^2_\rho=\{ v   \Big| \parallel  v \parallel_{L^2_\rho}^2 \equiv \int_{-1}^{1}| v(x)|^2 \rho\;dy< +\infty \} ,
 \end{eqnarray}
and
 \begin{eqnarray}\label{H1}
H^1_\rho=\{ v \, \Big|  \parallel  v \parallel_{L^2_\rho} + \parallel   \nabla v \parallel_{L^2_\rho} < +\infty \}.
 \end{eqnarray}
We see that $E$ is well defined from the fact that the three first terms of its expression in (\ref{15}) are in $L^1_\rho$ ; for the last term we need to use the Hardy-Sobolev inequality given by Merle and Zaag in Appendix B page $1163$ of \cite{MZ05}:

$$ \parallel  w \parallel_{L^{p+1}_\rho}\le C\parallel w \parallel_{H^1_\rho}.$$

Now, if $u$ is a solution of (\ref{equa}), with blow-up surface $\Gamma \,:\, \{x\rightarrow T(x)\},$ and if $x_0\neq 0$, then we have the following:

\begin{theorr}{\label{T0}}{\bf(Bound in similarity variables outside the origin)} 

\noindent $i)${\bf (Non-characteristic case)}:

 If $x_0\neq 0$ is a non-characteristic point,  then, for all $s$ large enough:

$$0 < \epsilon_0(p)\le||w_{x_0}(s)||_{H^1(-1,1)}+||\partial_s w_{x_0}(s)||_{L^2(-1,1)}\le K(x_0).$$
$ii)${\bf (Characteristic case)}:\\
If $x_0\neq 0$ is a characteristic point,  then, for all $s$ large enough:
$$||w_{x_0}(s)||_{H^1_\rho}+||\partial_s w_{x_0}(s)||_{L^2_\rho}\le K(x_0).$$
\end{theorr}

Using the bound in Theorem \ref{T0}, together with the compactness procedure based on the existence of a Lyapunov for equation (\ref{equa}) (which is a perturbation of the functional $E(w,\partial_s w)$ defined in (\ref{15})), we derive the following:

\begin{theorr}{\label{T00}}{\bf(Blow-up behavior in similarity variables outside the origin)} 
\\
$i)$  {\bf (Non-characteristic case)} 
The set $\mathcal{R}\cap \mathbb{R}_+^*$
 is open, and $T$ is of class $C^1$ on that set. Moreover, there exist $\mu_0>0$ and $C_0>0$ such that for all $x_0\in \mathcal{R}\cap \mathbb{R}_+^*$, there exist $\theta (x_0) =\pm 1$ and $s_0(x_0)\ge -\log T(x_0)$ such that for all $s\ge s_0$:
 \begin{eqnarray}\label{31}
\Big|\Big|\begin{pmatrix} w(s)\\\partial_s w(s) \end{pmatrix} - \theta(x_0) \begin{pmatrix} \kappa(T'(x_0)\sqrt{a(x_0)})\\0\end{pmatrix} \Big|\Big|_{{H^1_{\rho}\times L^2_\rho} }\le C_0 e^{-\mu_0(s-s^*)},
 \end{eqnarray}
 where $w=w_{x_0}$. Moreover, $E(w,\partial_s w)\rightarrow E(\kappa_0,0) $ as $s\rightarrow \infty.$\\
 $ii)$ {\bf (Characteristic case)} If $x_0\in \mathcal{S}\cap \mathbb{R}_+^*$, there is $\xi_0 (x_0) \in \mathbb{R}$ such that:
  \begin{eqnarray}\label{4.3}
\Big|\Big|\begin{pmatrix} w(s)\\\partial_s w(s) \end{pmatrix} -\theta_1 \begin{pmatrix} \sum_{i=1}^{k(x_0)}      (-1)^{i+1}\kappa(\hat d_i (s),.)\\0\end{pmatrix} \Big|\Big|_{H^1_{\rho}\times L^2_\rho}\rightarrow 0,
 \end{eqnarray}
 where $w=w_{x_0}$ and $E(w,\partial_s w)\rightarrow k(x_0) E(\kappa_0,0)  $ as $s\rightarrow \infty$, for some $k(x_0) \ge 2$, $\theta_i=\theta_1 (-1)^{i+1}$, $\theta_1=\pm 1$, and continuous $\hat d_i (s)=-\tan \hat \xi_i (s)$ with 
 \begin{eqnarray}\label{4.6}\hat  \xi_i (s)=\bar  \xi_i (s)+  \xi_0,
  \end{eqnarray}
 where $\bar  \xi_i(s)$ is introduced in (\ref{4.1}).
\end{theorr}
\noindent {\bf Remark:} Estimate (\ref{31}) holds in $H^1\times L^2 (-1,1)$, thanks to the covering argument introduced by Merle and Zaag in \cite{MR2147056}. From the Sobolev embedding, it holds also in $L^\infty\times L^2$. \\
\noindent {\bf Remark:} Following the strategy of C\^{o}te and Zaag in \cite{CZ12}, refined in \cite{HZ2019} by Hamza and Zaag, for every $k_0\ge 2$ and $\xi_0\in \mathbb{R}$, 
 we are able to construct examples of solutions to equation (\ref{waveq}) showing a characteristic-point and obeying the modality described in item $ii)$ of Theorem \ref{T00}.
\medskip

Going back to $u(x,t)$ thanks to (\ref{nchg}), we have the following corollary:
\begin{corr}{\bf (Blow-up profile for equation (\ref{waveq}) in the non-characteristic case outside the origin) }$\; $\\
If $\mathcal{R}\cap \mathbb{R}_+^*$, then we have
 \begin{eqnarray*}
u(x,t)   \sim   \frac{\theta(x_0) \kappa_0 (1-a(x_0)|T'(x_0)|^2)^{\frac{1}{p-1}}}{ (T(x_0)-t+T'(x_0)\sqrt{a(x_0)}(\phi(x)-\phi(x_0))^\frac{2}{p-1} }\mbox{ as } t \rightarrow T(x_0)
\end{eqnarray*}
 uniformly for $ x$ such that $ |\phi(x)-\phi(x_0) |< T(x_0)-t.$
\label{cor3}
\end{corr}
We also obtain the regularity of the blow-up set:
\begin{propp}\label{prop0}{\bf(Regularity of the blow-up set outside the origin)}$\,$\\\label{blow-upset}
$i)$  {\bf (Non-characteristic case)} It holds that $\mathcal{R}\neq \emptyset $, $\mathcal{R}  \backslash \{0\}$ is an open set, and $x\mapsto T(x)$ is of class $C^1$ on $\mathcal{R} \backslash \{0\}$ and for all $x\in \mathcal{R} \backslash \{0\}$, $| T'(x)   |<\frac{1}{\sqrt{a(x)}}$.\\
$ii)$ {\bf (Characteristic case)} Any $x_0\in \mathcal{S} \backslash \{0\}$ is isolated. In addition, if $x_0 \in  \mathcal{S} \backslash \{0\}$ with $k(x_0)$ solitons and $\xi_0(x_0) \in \mathbb{R}$ as center of mass of the solitons' center as shown in (\ref{4.3}) and (\ref{4.6}), then
 \begin{align}\label{6.1}
T'(x)+ \frac{\theta (x)}{\sqrt{a(x)}} &\sim \frac{\theta (x) \nu e^{-2 \theta (x) \xi_0 (x_0) } 
}{\sqrt{a(x_0)}  |\log |x-x_0||^\frac{(k(x_0)-1)(p-1)}{2}},
 \\ \label{6.2}
T(x)-T(x_0)+ |\phi(x)-\phi(x_0)| &\sim\frac{\nu  |\phi(x)-\phi(x_0)|e^{-2  \theta (x)  \xi_0 (x_0) } }{  |\log |x-x_0||^\frac{(k(x_0)-1)(p-1)}{2}},
\end{align}
 as $x\rightarrow x_0$, where $\theta(x)=\frac{x-x_0}{ |x-x_0| }$ and $\nu=\nu (p)>0$.
\end{propp}
\noindent{\bf Remark:} If $a$ is Holder continuous, then we may replace $\frac{\theta (x)}{\sqrt{a(x)}}$ by $\frac{\theta (x)}{\sqrt{a(x_0)}}$ in (\ref{6.1}), and replace  (\ref{6.2}) by
 \begin{align}
 T(x)-T(x_0)+  \frac{|x-x_0|}{\sqrt{a(x_0)}} &\sim\frac{ \nu  \frac{| x-x_0|}{\sqrt{a(x_0)}}e^{-2  \theta (x)  \xi_0 (x_0)} }{  |\log |x-x_0||^\frac{(k(x_0)-1)(p-1)}{2}},
\end{align}

\subsubsection{Behavior at the origin}
When $x_0=0$, we have $\phi(x_0)=0$, hence the term $\frac{e^{-s}(d-1)}{\phi(x_0)+y e^{-s}}\partial_y w_{_0}=\frac{d-1}{y}\partial_y w_{_0}$ in equation (\ref{equa}) and can no longer be treated as a perturbation.\\
Accordingly, we may write the second and first order space derivatives in the following divergence form:
\begin{equation*} 
(1-y^2)\partial_y^2 w_{_0}-2 \frac{p+1}{p-1} y \partial_y w_{_0}+\frac{d-1}{y}\partial_y w_{_0}=\frac{1}{\rho_{_0}(y)}  \partial_y (\rho_{_0}(1-y^2)\partial_y w_{_0})
\end{equation*}
where 
\begin{equation} \label{rho}
\rho_{_0}(y)=(1-y^2)^{\frac{2}{p-1}-\frac{d-1}{2}} y^{d-1}.
\end{equation}

\medskip

For the case where $(f,g)\equiv(0,0)$,  in one space dimension, we introduce the functional 
\begin{equation}\label{1500}
 E_{_0}(w,\partial_s w)=\int_{0}^{1} \left( \frac{1}{2} |\partial_s w|^2+\frac{1}{2} |\partial_y w|^2 (1-y^2)+\frac{p+1}{(p-1)^2}|w|^2-\frac{\beta(0)}{p+1}|w|^{p+1}\right) \rho_{_0} dy.
\end{equation}

Note first that $E_{_0}$ is defined if $(w, \partial_s w)\in H^1_{\rho_{_0}}\times L^2_{\rho_{_0}}$, where the norms $L^2_{\rho_{_0}}$ and $H^1_{\rho_{_0}}$ are defined by the same way as in (\ref{L2}) and (\ref{H1}), but only on the domain $(0,1)$ and with weight $\rho_{_0}$ given in (\ref{rho}). 

Adapting the techniques introduced by Antonini and Merle (See Section $2$ page $1144$ in \cite{AM01}) to our case where $\rho_{_0}$ is given by (\ref{rho}), we see that 
\begin{equation*}\label{}
 \frac{d}{ds}E_{_0}(w,\partial_s w)=(d-1-\frac{4}{p-1})\int_{0}^{1} (\partial_s w)^2\frac{\rho_0}{1-y^2} dy,
\end{equation*}
as $d$ satisfies (\ref{3.3})-(\ref{3.5}), $E_{_0}$ (\ref{1500}) is decreasing and is a Lyapunov functional. Another way to justify this: the functional in (\ref{1500}) is simply the radial version of the functional of \cite{AM01}
considered in the space $\mathbb{R}^d$ (which is not the physical space $\mathbb{R}^N$).\\

Considering $w(y,s)$ as a (non necessarily radial) function defined in $\mathbb{R}^d$, we may use the perturbative
techniques of Hamza and Zaag in \cite{HZ12} to derive the following:
\begin{theorr}{\label{thorigine}}{\bf(Bound in similarity variables at the origin in the non-characteristic case)} 
 If $u$ is a solution of (\ref{equa}) with blow-up surface $\Gamma \,:\, \{x\rightarrow T(x)\},$ and if $0$ is a non-characteristic point,  then, for $s$ large enough:
$$0 < \epsilon_0(p)\le||w_{_0}(s)||_{H^1_{\rho_{_0}}}+||\partial_s w_{_0}(s)||_{L^2_{\rho_{_0}}}\le K,$$
$$||w_{_0}(s)||_{H^1_{r^{d-1}}(0,1)}+||\partial_s w_{_0}(s)||_{L_{r^{d-1}}^2(0,1)}\le K.$$
\end{theorr}

\medskip

\subsection{ Strategy of the proof of the results}
Thanks to the transformation (\ref{6.5}), we reduce to the case where $a(x) \equiv 1$ in the remaining part of the paper. In comparison with the paper by Hamza and Zaag \cite{HZ}, our equation allows a non-constant term in front of the reaction-term $|u|^{p-1}u$, namely $\beta(x) \not \equiv 1$. As in \cite{HZ}, the most delicate point is to obtain a Lyapunov functional in similarity variables defined in (\ref{nchg}). Thus, in the following section, we focus on the proof of the existence of a Lyapunov functional for equation (\ref{equa}) in the first subsection, then we give some hints on how to adapt the strategy of \cite{HZ} to derive the blow-up behavior outside and at the origin in the second and third subsections.\\

\noindent{\bf Acknowledgments}\\
The authors would like to thank Professor Mohamed Ali Hamza, for his helpful advices during the preparation of this paper, which greatly improved the presentation of the results.

\noindent This material is based upon work supported by Tamkeen under the NYU Abu Dhabi Research Institute grant CG002.

\section{Proof of the results}
We prove the blow-up results for (\ref{eq6.5}) which we recall in the following:
\begin{equation} \left\{
\begin{array}{l}
\displaystyle\partial^2_{t}U =\partial^2_{X}U+\frac{d-1}{X}\partial_{X}U +\beta(X) |U|^{p-1}U+f(U)+G(X,t,\partial_{X}U,\partial_t U) , \mbox{ for } X>0\\
U_{X}(0,t) = 0,\\
U(0)=U_{0} \mbox{ and }  U_{t}(0) = U_{1},
\end{array}
\right . \label{}
\end{equation}

with
\begin{align*}
|f(U)|&\le M(1+|U|^q ),\mbox{  for all }U\in\mathbb{R}\mbox{  with }  (q<p, M>0),\\
|G(X,t,\partial_{X}U,\partial_t U)  |&\le M(1+|\partial_{X}U|+|\partial_t U|).
\end{align*}

In fact, this is almost the same equation as in \cite{HZ} except for the coefficient $\beta(X)$ in front of $|U|^{p-1}U$ which was taken identically equal to $1$ in \cite{HZ}. For that reason, we follow the strategy of \cite{HZ} and focus mainly on the treatment of the term $\beta(X) |U|^{p-1}U$. Given some $X_0=\phi(x_0) \in \Bbb R_+,$ where $\phi$ was defined in (\ref{phi}), we introduce the following self-similar change of variables, as in (\ref{nchg}):
\begin{equation}\label{trans_auto}
w_{X_0}(y,s) =(T_U(X_0)-t)^\frac{2}{p-1}U(X,t), \quad  y=\frac{X-X_0}{T_U(X_0)-t}, \quad
s=-\log(T_U(X_0)-t).
\end{equation}
Note that the curve $T_U$ of $U$ is given by the curve $T$ of $u$, in fact :
$$T_U(X)=T(x)\mbox { with } X=\phi(x).$$
This change of variables transforms the backward light cone with vertex $(X_0, T_U(X_0))$ into the infinite cylinder $(y,s)\in (-1,1) \times [-\log T_U(X_0),+\infty).$ The function $w_{X_0}$ (we write $w$ for simplicity) satisfies the following equation for all $|y|<1$ and $s\ge -\log T_U(X_0)$:

\begin{eqnarray} \label{eqwbeta}
 \partial^2_{s} w &=&(1-y^2)\partial_y^2 w-2 \frac{p+1}{p-1} y \partial_y w-2 \frac{p+1}{(p-1)^2}w+ \beta (X_0)|w|^{p-1}w-\frac{p+3}{p-1} \partial_s w- 2 y \partial_{ys} w\notag\\ &+& e^{-s} \frac{(d-1)}{X_0+y e^{-s}} \partial_y w+  e^{-\frac{2ps}{p-1}}f(e^{\frac{2s}{p-1}}w)+
 (\beta (X_0+y e^{-s})-\beta(X_0))|w|^{p-1}w\\ &+& e^{-\frac{2ps}{p-1}}  G(X_0+y e^{-s},T_0-e^{-s},e^{\frac{(p+1)s}{p-1}} \partial_y w,  e^{\frac{(p+1)s}{p-1}} ( \partial_{s} w+y\partial_{y} w+ \frac{2}{p-1} w )).\notag
 \end{eqnarray}

In the whole paper, we use the notation
 \begin{eqnarray}\label{F}
F(u)=\int_{0}^{u}f(v)dv.
 \end{eqnarray}

\subsection{A Lyapunov functional in similarity variables outside the origin}\label{2.1}
In this subsection we prove the existence of a Lyapunov functional and the novelty lays in the new coefficient $\beta(X)\not\equiv 1$.
We recall that for the case $(f,G)\equiv(0,0)$ with a constant $\beta$, the Lyapunov functional in one space dimension is
\begin{equation}\label{150}
 E_0(w,\partial_s w)=\int_{-1}^{1} \left( \frac{1}{2} |\partial_s w|^2+\frac{1}{2} |\partial_y w|^2 (1-y^2)+\frac{p+1}{(p-1)^2}|w|^2-\frac{\beta(X_0)}{p+1}|w|^{p+1}\right) \rho dy.
\end{equation}
In order to find a Lyapunov functional for our equation (\ref{eqwbeta}), we introduce 
\begin{equation}\label{defE}
 E(w,\partial_s w)= E_0(w,\partial_s w)+I(w(s),s)+ J(w(s),s)+ K(w(s),s),
\end{equation}
where
\begin{equation}\label{i}
I(w(s),s)=-e^{-\frac{2(p+1)s}{p-1}}\int_{-1}^{1} F(e^{\frac{2s}{p-1}}w) \rho dy,\end{equation}
\begin{equation}\label{j}
 J(w(s),s)=   -\frac{1}{p+1}\int_{-1}^{1} (\beta(X_0+y e^{-s})-\beta(X_0))|w|^{p+1} \rho dy
\end{equation}
\begin{equation}\label{k}
K(w(s),s)= -e^{-\gamma s}\int_{-1}^{1} w \partial_s w \rho dy, \end{equation}
with
\begin{equation}\label{gamma}
\gamma =\min \left(\frac{1}{2}, \frac{p-q}{p-1}\right)>0.
\end{equation}
Then, we claim the following:
\begin{prop}\label{lyapunov}{\bf (Energy estimates outside the origin)}\\
$(i)$ There exist $C=C(p,M)>0$ and $S_0\in \mathbb{R}$ such that for all $X_0>0$ and for all $s\ge \max (-\log T_U(X_0), S_0, -4 \log X_0, -\log \frac{X_0}{2})$,
\begin{equation*}
\frac{d}{ds} E(w(s),s )\le \frac{p+3}{2} e^{-\gamma s} E_0(w(s), s) - \frac{3}{p-1} \int_{-1}^{1} ( \partial_s w)^2\frac{\rho}{1-y^2} dy+ C e^{-2 \gamma s}.
\end{equation*}
$(ii)$ There exists $S_1(p,N,M,q)\in \mathbb{R}$ such that, for all $s\ge \max(s_0,S_1)$, we have $H(w(s),s)\ge 0$.
\end{prop}
\noindent{\bf Remark:} From $(i)$, we see that $H$ given by
\begin{equation*}
 H(w(s),s )= E(w(s),s ) e^{\frac{p+3}{2\gamma} e^{-\gamma s}}+ \mu e^{-2\gamma s} (\mu>0)
\end{equation*}
is a Lyapunov functional for equation (\ref{eqwbeta}).

\begin{proof}[Proof of Proposition \ref{lyapunov}]$ $\\
In this proof we use the notation, $x_+=\max(0,x)$.\\
$(i)$ We proceed like Hamza and Zaag in \cite{HZ} (See page $1092$) and we deal with the new term coming from (\ref{eqwbeta}). For that reason, we give the equations, recall the estimates already proved in \cite{HZ} and focus only on the new term.\\

\noindent We multiply equation (\ref{eqwbeta}) by $\partial_s w \rho$ and integrate for $y\in (-1,1)$, using (\ref{i}) and  (\ref{j}), we have for $X=X_0+ye^{-s}$ :
\begin{eqnarray}\label{2.14}
&\,&\frac{d}{ds} (E_0(w(s),s)+I(w(s),s)+J(w(s),s))\notag \\
&=& \frac{-4}{p-1} \int_{-1}^{1} \frac{(\partial_s w)^2}{1-y^2} \rho dy+ \underbrace{ (N-1)e^{-s} \int_{-1}^{1}    \partial_s w \partial_y w \frac{\rho}{X} dy}_{ I_1(s)}   \notag \\
&+& \underbrace{\frac{2(p+1)}{p-1}  e^{-\frac{2(p+1)s}{p-1}} \int_{-1}^{1} F(e^{\frac{2s}{p-1}}w)  \rho dy}_{ I_2(s)}+ \underbrace{ \frac{2}{p-1}  e^{-\frac{2ps}{p-1}} \int_{-1}^{1} f(e^{\frac{2s}{p-1}}w)w  \rho dy}_{ I_3(s)} \\
 &+&\underbrace{\frac{e^{-s}}{p+1} \int_{-1}^{1} y\beta'(X_0+y e^{-s}) |w|^{p+1}  \rho dy}_{ I_4(s)}\notag  \\
&+& \underbrace{ e^{-\frac{2ps}{p-1}} \int_{-1}^{1} G(X_0+y e^{-s},T_0-e^{-s},    e^{\frac{(p+1)s}{p-1}} \partial_{y} w  , e^{\frac{(p+1)s}{p-1}} ( \partial_{s} w+y\partial_{y} w+ \frac{2 }{p-1} w ))   \partial_s w \rho dy}_{ I_5(s)} \notag 
\end{eqnarray}
The terms $I_1$, $I_2$ and $I_3$ can be controlled exactly as in page $1092$ in \cite{HZ}. For $I_5$,  comparing to the previous work, we see that $G$ involves new terms, but as it satisfies condition (\ref{Hg}), we can also adapt the study of Hamza and Zaag in \cite{HZ} to get :
\begin{eqnarray}\label{2.15}
|I_1(s)|& \le& C e^{-s}\int_{-1}^{1}( \partial_y w)^2 \rho (1-y^2) dy+ \frac{Ce^{-s}}{X_0^2}  \int_{-1}^{1} (\partial_s w)^2\frac{\rho}{1-y^2}  dy,\\
\label{2.16}
|I_2(s)|+|I_3(s)|& \le& C e^{-\frac{2(p-q)s}{p-1}}+C e^{-\frac{2(p-q)s}{p-1}}  \int_{-1}^{1}  | w|^{p+1}  \rho dy,\\
\label{2.17}
|I_5(s)|& \le& C e^{-s}\int_{-1}^{1}\left(( \partial_y w)^2 (1-| y|^2) + \frac{(\partial_s w)^2}{1-y^2}+ w^2\right) \rho  dy+Ce^{-s}.
\end{eqnarray}
For the new term $I_4$, we use the fact that $\beta$ is of class $C^1$, we get:
\begin{eqnarray}\label{2.18}
|I_4 (s)| \le     \frac{e^{-s}}{p+1}  || \beta'  ||_{L^\infty_{((X_0-T)_+, X_0+T)}}  \int_{-1}^{1}   |w|^{p+1}    \rho  dy.
\end{eqnarray}
Using (\ref{2.14}), (\ref{2.15}), (\ref{2.16}), (\ref{2.17}) and (\ref{2.18}), we have
\begin{eqnarray}\label{2.20}
&\,& \frac{d}{ds} (E_0(w(s),s)+I(w(s),s)+J(w(s),s)) \le  (\frac{-4}{p-1}+
 C e^{-\frac{s}{2}} ) \int_{-1}^{1} (\partial_s w)^2\frac{\rho}{1-y^2}  dy \notag
\\&+&C e^{-s} \int_{-1}^{1} \left( ( \partial_y w)^2  (1-|y|^2)+ w^2\right) \rho dy+ C e^{-2 \gamma  s} \int_{-1}^{1} |w|^{p+1}  \rho dy+C e^{-2\gamma s}.
\end{eqnarray}
Now, we consider $(K(w(s),s))$ (\ref{k}). Using equation (\ref{eqwbeta}) and integration by parts, we write:
\begin{eqnarray*}\label{}
&\,&e^{ \gamma  s}  \frac{d}{ds} (K(w(s),s)) =-\int_{-1}^{1} (\partial_s w)^2\rho dy+ \int_{-1}^{1} (\partial_y w)^2 (1-y^2)\rho dy+\frac{2p+2}{(p-1)^2} \int_{-1}^{1} w^2\rho dy\\
&-&\beta(X_0)\int_{-1}^{1} |w|^{p+1} \rho dy+(\gamma+\frac{p+3}{p-1} -2N )\int_{-1}^{1}w \partial_s w\rho dy-2\int_{-1}^{1}w \partial_s w y\rho' dy\\
&-& 2 \int_{-1}^{1}\partial_s w \partial_y w y\rho dy- e^{-\frac{2ps}{p-1}}\int_{-1}^{1} w f\left(e^{\frac{2s}{p-1}} w\right) \rho dy-(N-1)  e^{-s}\int_{-1}^{1}w \partial_y w \frac{\rho }{r}   dy\\
 &-&e^{-\frac{2ps}{p-1}}\int_{-1}^{1} w G(X_0+y e^{-s},T_0-e^{-s},  e^{\frac{(p+1)s}{p-1}} \partial_{y} w  , e^{\frac{(p+1)s}{p-1}} ( \partial_{s} w+y\partial_{y} w+ \frac{2 }{p-1} w )) \rho dy\\
 &-&\int_{-1}^{1}(\beta(X_0+y e^{-s})-\beta(X_0)) |w|^{p+1} \rho dy.
\end{eqnarray*}
Using (\ref{i}) and (\ref{j}), we get:
\begin{eqnarray}\label{j2}
&\,&e^{ \gamma  s}  \frac{d}{ds} (K(w(s),s)) =\frac{p+3}{2} (E_0(w(s))+I(w(s))+ J(w(s)) )- \frac{p-1}{4}  \int_{-1}^{1} (\partial_y w)^2 (1-y^2)\rho dy\notag\\
&-&\frac{p+1}{2(p-1)} \int_{-1}^{1} w^2\rho dy-\frac{p-1}{2(p+1)}\beta(X_0)\int_{-1}^{1} |w|^{p+1} \rho dy\notag\\
&+&\underbrace{(\gamma+\frac{p+3}{p-1} -2N +\frac{p+3}{2} e^{- \gamma  s}  )\int_{-1}^{1}w \partial_s w\rho dy}_{ K_1(s)}\notag \\
&+&\underbrace{\frac{8}{p-1}\int_{-1}^{1}w \partial_s w  \frac{y^2}{1-y^2}\rho dy}_{ K_2(s)}
\underbrace{-2\int_{-1}^{1}\partial_s w \partial_y w y\rho dy}_{ K_3(s)}
\underbrace{-e^{-\frac{2ps}{p-1}}\int_{-1}^{1}w  f(e^{\frac{2ps}{p-1}}w)\rho dy}_{ K_4(s)}\notag\\
 &-&\underbrace{e^{-\frac{2ps}{p-1}}\int_{-1}^{1} w G(X_0+y e^{-s},T_0-e^{-s},  e^{\frac{(p+1)s}{p-1}} \partial_{y} w  , e^{\frac{(p+1)s}{p-1}} ( \partial_{s} w+y\partial_{y} w+ \frac{2 }{p-1} w )) \rho dy}_{ K_5(s)}\notag\\
 &+&\underbrace{\frac{p+3}{2}e^{-\frac{2(p+1)s}{p-1}} \int_{-1}^{1} F(e^{\frac{2}{p-1}s}w)\rho dy}_{ K_6(s)}
 \underbrace{-(N-1)e^{-s}\int_{-1}^{1} w \partial_y w\frac{ \rho}{r} dy}_{ K_7(s)}\notag\\
 &-&\underbrace{\frac{p-1}{2(p+1)}\int_{-1}^{1}(\beta(X_0+y e^{-s})-\beta(X_0)) |w|^{p+1} \rho dy}_{ K_8(s)}
\end{eqnarray}

Note that all the terms $K_1$, $K_2$, $K_3$, $K_4$, $K_5$, $K_6$ and $K_7$ have been studied  in \cite{HZ} (for details see page $1094$ in \cite{HZ}). For the reader's convenience, we recall the following estimates:
\begin{eqnarray}\label{2}
|K_1(s)|& \le& C e^\frac{\gamma s}{2}\int_{-1}^{1}  (\partial_s w)^2\frac{\rho}{1-y^2}  dy+C   e^{-\frac{\gamma s}{2} }   \int_{-1}^{1} w^2\rho  dy ,\\
|K_2(s)|& \le&  C e^\frac{\gamma s}{2}\int_{-1}^{1}  (\partial_s w)^2\frac{\rho}{1-y^2}  dy+C   e^{-\frac{\gamma s}{2} }   \int_{-1}^{1} w^2\rho  dy\notag\\
&+&C   e^{-\frac{\gamma s}{2} }   \int_{-1}^{1} ( \partial_y w)^2 \rho (1-y^2) dy,\\
|K_3(s)|& \le&  C e^\frac{\gamma s}{2}\int_{-1}^{1}  (\partial_s w)^2\frac{\rho}{1-y^2}  dy+C   e^{-\frac{\gamma s}{2} }   \int_{-1}^{1} ( \partial_y w)^2 \rho (1-y^2) dy,\\
|K_4(s)|+|K_6(s)|& \le&  C e^{-\gamma s}+C e^{-\gamma s} \int_{-1}^{1}  | w|^{p+1}  \rho dy\\
|K_5(s)|& \le&  C e^{-\gamma s }\int_{-1}^{1}  (\partial_s w)^2\frac{\rho}{1-y^2}  dy+C e^{-\gamma s }  \int_{-1}^{1} ( \partial_y w)^2 \rho (1-y^2) dy\notag\\
&+&C e^{-\gamma s }  \int_{-1}^{1}  w^2\rho  dy+C e^{-\gamma s }\\   
|K_7(s)|& \le&  C e^{-\frac{s}{2} }\int_{-1}^{1} ( \partial_y w)^2 \rho (1-y^2) dy+ C e^{-\frac{s}{2} }\int_{-1}^{1} w^2\rho  dy.
\end{eqnarray}

For the new term $K_8$, using the fact that $\beta$ is of class $C^1$ we see that:
\begin{eqnarray}\label{k8}
|K_8 (s)| \le  \frac{p-1}{2(p+1)}         || \beta  ||_{L^\infty((X_0-T)_+, X_0+T)}   e^{-s}   \int_{-1}^{1}   | w|^{p+1}    \rho  dy.
\end{eqnarray}
By the same way, using the fact that $\beta$ is of class $C^1$, we prove that $J$ (\ref{j}) satisfies:
\begin{eqnarray}\label{k88}
|J(w(s))| \le Ce^{-s}    \int_{-1}^{1}   | w|^{p+1}    \rho  dy. 
\end{eqnarray} 
Using the definitions of $I$ (\ref{i}), $F$(\ref{F}) and the condition (\ref{condition}) we see that:
\begin{align}\label{k89}
|I(w(s))|& =\big | C e^{-2 \frac{p+1}{p-1}s} \int_{-1}^1  \int_{0}^{ \frac{2s}{p-1}w} f(v) dv  \rho  dy\big | \notag\\
&\le Ce^{ \frac{-2ps}{p-1}}   \int_{-1}^{1}    |w |    \rho  dy +Ce^{ \frac{-2(p-q)s}{p-1}} \int_{-1}^{1} | w|^{q+1}\rho  dy\\
&\le Ce^{ \frac{-2(p-q)s}{p-1}}   +Ce^{ \frac{-2ps}{p-1}}   \int_{-1}^{1}   w^2    \rho  dy +Ce^{ \frac{-2(p-q)s}{p-1}}   \int_{-1}^{1}   | w|^{p+1}    \rho  dy.  \notag
\end{align}

Using (\ref{j2})-(\ref{k89}) and the definition of $\gamma $ \eqref{gamma} we deduce that
\begin{eqnarray}\label{2.29}
&\,&e^{ \gamma  s}  \frac{d}{ds} (K(w(s),s)) \le \frac{p+3}{2} E_0(w(s)) \notag\\
&+&\left(C e^{-\frac{\gamma s}{2} }-\frac{p-1}{4} \right) \int_{-1}^{1} (\partial_y w)^2 (1-y^2)\rho dy
+\left(C e^{-\frac{\gamma s}{2}} - \frac{p+1}{2(p-1)}\right)  \int_{-1}^{1} w^2\rho dy\\
&+&\left(C e^{-\frac{\gamma s}{2}} - \frac{p+1}{2(p-1)}\beta(X_0)\right) \int_{-1}^{1} |w|^{p+1} \rho dy
+C e^\frac{\gamma s}{2} \int_{-1}^{1}(\partial_s w)^2  \frac{\rho }{1-y^2}dy+C e^{-\gamma s}.\notag
\end{eqnarray}

Using the definition \eqref{defE} of $E$, (\ref{2.20}), (\ref{2.29})  we get
(remember from \eqref{gamma} that $\gamma\le \frac 12$)
\begin{eqnarray*}\label{}
\frac{d}{ds} (E(w(s),s))& \le& C e^{-2\gamma s}+\frac{p+3}{2}e^{ -\gamma  s}  E_0(w(s),s)
+\left(C e^{-\frac{\gamma s}{2} }-\frac{4}{p-1} \right)  \int_{-1}^{1}(\partial_s w)^2  \frac{\rho }{1-y^2}dy\notag\\
&+&\left(C e^{-\frac{\gamma s}{2}} - \frac{p+1}{2(p-1)}\right)e^{ -\gamma  s}   \int_{-1}^{1} w^2\rho dy\\
&+&\left(C e^{-\frac{\gamma s}{2} }-\frac{p-1}{4} \right)e^{ -\gamma  s} \int_{-1}^{1} (\partial_y w)^2 (1- |y|^2)\rho dy\notag\\
&+&\left(C e^{-\frac{\gamma s}{2}} - \frac{p+1}{2(p-1)} \beta(X_0)\right) e^{ -\gamma  s}\int_{-1}^{1} |w|^{p+1} \rho dy.\notag\\
\end{eqnarray*}
Then, for $S_0$ well chosen large enough so that $s\ge \max (-\log T(X_0), S_0, -4 \log X_0, -\log \frac{X_0}{2})$, we write
\begin{equation*}
\frac{d}{ds} E(w(s),s )\le \frac{p+3}{2} e^{-\gamma s} E_0(w(s), s) - \frac{3}{p-1} \int_{-1}^{1} ( \partial_s w)^2\frac{\rho}{1-y^2} dy+ C e^{-2 \gamma s}.
\end{equation*}
This yields item $(i)$ of Proposition \ref{lyapunov}.\\
$(ii)$ 
This follows from the blow-up criterion proved by Antonini and Merle in \cite{AM01}. In fact, we need to follow the perturbative argument of Hamza and Zaag \cite{HZ}. As in \cite{HZ}, it is easy to prove the following identity for large $s$ and for any $w\in \mathcal{H}$:
$$H(w)\ge -\frac{2\beta (X_0)}{p+1}\int_{-1}^{1}  |w|^{p+1}\rho dy.$$ 
For more details see $ii)$ page $1096$ in \cite{HZ} and see page $1147$ in  \cite{AM01}. This concludes the proof of Proposition \ref{lyapunov}.
\end{proof}
\subsection{Blow-up results outside the origin}\label{2.2}

In this subsection, we give the main ideas of the proofs of our blow-up results outside the origin (Theorem \ref{T0}, Theorem \ref{T00}, Corollary  \ref{cor3} and Proposition \ref{prop0}). However, we will not give the details.
  In fact, thanks to the transformation (\ref{6.5}), we obtain the equation (\ref{eq6.5}) which is almost the same equation already studied by Hamza and Zaag in \cite{HZ}. In addition, in Subsection \ref{2.1}, we see that a Lyapunov functional is available (see the remark following Proposition \ref{lyapunov}), so with this informations, the reader can easily see that  the strategy adapted in \cite{HZ} from the strategy developed by Merle and Zaag in \cite{MZ05}, \cite{MR2147056}, \cite{MR2362418}, \cite{MR2415473}, \cite{MR2931219} and \cite{Mz12} together with C\^{o}te and Zaag \cite{CZ12} holds with very minor adaptations (See also \cite{MR2799813}). For that reason,  we will sketch the main steps in the following and explicit only the delicate estimates :\\
 
- In Step 1, we show that the solution is bounded in self-similar variables in the energy norm. In particular, we will prove Theorem \ref{T0}.\\

- In Step 2, we find the asymptotic behavior and derive the regularity of the blow-up curve. In particular, we will prove Theorem \ref{T00}, Corollary  \ref{cor3} and Proposition \ref{prop0}.\\

\medskip

\noindent{\bf Step 1 : Boundedness of the solution in similarity variables}

\noindent We derive with no difficulty the following:
\begin{prop}\label{prp}
For all $X_0>0$, there is a $C_2 (X_0)>0$ and $S_2 (X_0)\in \mathbb{R}$ such that for all $X\in [\frac{X_0}{2}, \frac{3X_0}{2}]$ and $s\ge S_2(X_0)$,
\begin{equation*}\label{}
\int_{-1}^{1} \left( \partial_y w^2(1-y^2) +w^2+ \partial_s w^2 +\beta (X_0+ye^{-s})|w|^{p+1}\right) \rho dy\le C_2 (X_0).
\end{equation*}
\end{prop}
\begin{proof}[Proof of Proposition \ref{prp}]

The adaptation by Hamza and Zaag in page 1091 in \cite{HZ} to the perturbed case works in our case ($\beta (X_0+ye^{-s})\not\equiv 1$) with no difficulty. As in \cite{HZ},  the adaptation is straightforward from \cite{MZ05} and Proposition $3.5$ page $66$ in \cite{MR2362418}.
\end{proof}

\begin{proof}[Proof of Theorem \ref{T0}]

$i)$ Consider $x_0>0$ and $X_0= \phi(x_0)>0$. Let us start with the upper bound. Since $b$ is continuous, $\beta$ is continuous too and as $\beta (X_0+ye^{-s})>0$ for $y\in(-1,1)$ and $s$ large enough., we derive from Proposition \ref{prp} that
\begin{equation*}\label{}
\int_{-\frac{1}{2}}^{\frac{1}{2}} \left( \partial_y w^2 +w^2+ \partial_s w^2 +|w|^{p+1}\right) dy\le C_3 (X_0).
\end{equation*}
Using the covering method of Proposition 3.4 in \cite{MR2147056}, we recover the desired upper bound. As for the lower bound, it follows exactly as in the unperturbed case in Lemma 3.1 in \cite{MR2147056}, simply because equation (\ref{eq6.5}) is well posed in $H^1\times L^2$, from \cite{GT} as we explained in the introduction.

$ii)$ It is a direct consequence of Proposition \ref{prp}.
\end{proof}
\medskip

\noindent{\bf Step 2 : Dynamics of the solution and properties of the blow-up curve}

\noindent We recall form the definition of $T$ (\ref{T}) that
 \begin{align}\label{T-phi}T(x)=T_U(\phi(x))=T_U(X),\mbox{ and } T_U'(X)=\sqrt{a(x)} T'(x).
\end{align}

\begin{proof}[Proof of Theorem \ref{T00}]$ $\\
 i) {\bf Non-characteristic case :}  As we said before, our equation (\ref{eq6.5}) is the same as in \cite{HZ}, except for the coefficient $\beta(X)\not \equiv 1$. Thanks to the Proposition \ref{prp} above, the adaptation of Hamza and Zaag of the analysis of \cite{MR2362418} and \cite{MR2415473} works ; in particular, this is the case for Theorem $1$ page $1097$ which gives the following profile of $w$

 \begin{align}\label{profil-w}
w(y,s ) 
  &\sim\theta(x_0) \kappa_0 \frac{ (1-T_U'(X_0)^2)^{\frac{1}{p-1}}}{ (1+T_U'(X_0)y)^\frac{2}{p-1} }\mbox{ as } s \rightarrow +\infty,
\end{align}
where $X_0=\phi(x_0)$, in $H^1(-1,1)$. Using (\ref{T-phi}), we get our statement.\\
 ii) {\bf Characteristic case :}  The approach of C\^ote and Zaag in \cite{CZ12}, adapted to the perturbed case by Hamza and Zaag in \cite{HZ} stays valid in our case (for more details see page 1105 in \cite{HZ}).
\end{proof}
\begin{proof}[Proof of Corollary  \ref{cor3}] Applying the transformation (\ref{nchg}) to the profile given in (\ref{profil-w}) and using the sobolev embedding, we see that for $X_0\in \mathcal{R}\cap \mathbb{R}_+^*$, we have

 \begin{eqnarray*}
U(X,t)   \sim   \frac{\theta(x_0) \kappa_0 (1-T_U'(X_0)^2)^{\frac{1}{p-1}}}{ (T_U(X_0)-t+T'_U(X_0)(X-X_0))^\frac{2}{p-1} }\mbox{ as } t \rightarrow T_U(X_0)
\end{eqnarray*}
 uniformly for $ X$ such that $ |X-X_0 |<  T_U(X_0)-t$. Applying (\ref{T-phi}), we get the result.
\end{proof}
\begin{proof}[Proof of Proposition \ref{prop0}]$ $

$i)$ We can easily see that the strategy developed in the non-perturbed case in \cite{MR2415473}, and then adapted to the perturbed case in \cite{HZ} works in the present case ($\beta(X)\not \equiv1$), with minor adaptations.

$ ii)$ 
Let $x_0 \in  \mathcal{S} \backslash \{0\}$ with $k(x_0)$ solitons and $\xi_0(x_0) \in \mathbb{R}$ as center of mass of the solitons as shown in (\ref{4.3}) and (\ref{4.6}). Proceeding as in the adaptation by Hamza and Zaag in Theorem 5 in \cite{HZ} to the perturbed case (see also Theorem 1 and 2 in \cite{Mz12}, where this statement was proved with no perturbation),
 so we get

 \begin{align*}\label{}
T_U'(X)+\theta_U (X) &\sim \frac{\theta_U (X) \nu e^{-2 \theta_U (X) \xi_0 (X_0) } 
}{  |\log |X-X_0||^\frac{(k_U(X_0)-1)(p-1)}{2}},
 \\ \label{}
T_U(X)-T_U(X_0)+ |X-X_0| &\sim\frac{\nu e^{-2  \theta_U (X)  \xi_0 (X_0) } |X-X_0| }{  |\log |X-X_0||^\frac{(k_U(X_0)-1)(p-1)}{2}},
\end{align*}
 as $X\rightarrow X_0$, where $ \theta_U (X)=\frac{X-X_0}{ |X-X_0| }$ and $\nu=\nu (p)>0$. Using the correspondance between $x$ and $X$ and also $T_U$ and $T$ shown in (\ref{6.5}) and (\ref{T}), we recover our conclusion.
\end{proof}

\medskip

\subsection{Blow-up results at the origin}\label{2.3}

\begin{proof}[Proof of Theorem \ref{thorigine}]
The proof is done in the framework of similarity variables (\ref{nchg}), with $x_0=0$. Since $d$ is an integer, one
clearly sees that the equation satisfied by $w_0$ in (\ref{equa}) is simply the radial version of the multi-dimensional
equation considered in $\mathbb{R}^d$. Since $p$ is subconformal in relation to $d$, as shown in (\ref{3.5}), we are in the setting
considered by Hamza and Zaag in \cite{HZ} for perturbed equations, with the exceptions that we have a non-constant
coefficient in front of the nonlinear term here. As we have already seen while investigating the Lyapunov functional
in Subsection \ref{2.1}, that is not an issue, and one can adapt the proof of \cite{HZ} to the present equation, with no difficulties.
\end{proof}

\appendix 
\section{$L^2_{loc,u}$ for radial functions  }
\label{appendixA}
Note that we handle only $L^2$-type spaces, since the extension to $H^1$-type spaces is natural.
Consider $u$ a radial solution in  $L_{loc,u}^2$ in $\mathbb{R}^d$ and introduce $\tilde u$ such that
$u(x)=\tilde u(r)$ with $r=|x|,\; \forall x\in \mathbb{R}^d$.

Let $A=\displaystyle\sup\limits_{x_0\in \mathbb{R}^d}
\int_{B(x_0,1)}|u(x)|^2 dx $ the square of the $L_{loc,u}^2$ norm in $\mathbb{R}^d$ and\\ $B=\displaystyle\sup\limits_{r_0\ge  1}\frac{1}{r_0^{d-1}}
\int_{r_0-1}^{r_0+1}|\tilde u(r)|^2 r^{d-1}dr .$
 \noindent We also define for the crown $\mathcal{C}(r_0,1)$ by$$\forall r_0\ge 1,\;\mathcal{C}(r_0,1)=\{ x\in \mathbb{R}^d, |\;\; r_0-1\le |x| < r_0+1\}.$$
 
We aim at proving that the square root of $B$ is an equivalent norm to the $L_{loc,u}^2$ in the radial setting. More precisely, we have the following: 
\begin{lem}$ $\\
$i)$ $\exists \bar \alpha (d)>0$ such that $A\le \bar \alpha (d) B$.\\
$ii)$ $\exists \bar \beta (d)>0$ such that $B\le \bar \beta (d) A$.
\end{lem}

\begin{proof}$ $\\
$i)$ It is enough to show that for any $x_0\in \mathbb{R}^d$, 
$$ \int_{B(0,2)}|u(x)|^2 dx \le    \bar \alpha (d) B,\mbox{ for some } \bar  \alpha (d) >0  . $$
Consider $x_0\in \mathbb{R}^d$. If $|x_0|<1$ and $x\in B(x_0,1)$ then $|x|<|x_0|+1<2$. Consequently,
\begin{eqnarray*}
\int_{B(x_0,1)}|u(x)|^2 dx\le \int_{B(0,2)}|u(x)|^2 dx= \omega_{d-1}\int_{0}^2 |\tilde u(r)|^2 r^{d-1}dr \le \omega_{d-1}B,
\end{eqnarray*}
where $\omega_{d-1}$ is the volume of the sphere $S^{d-1}$.\\
\medskip

 \noindent Now, if $|x_0|\ge1$, then we have $ B(x_0,1)     \subset \mathcal{C}(|x_0|,1)   $. Furthermore, for geometric considerations, we know that there exists $\alpha (d, |x_0|)>0$ such that the crown $\mathcal{C}(|x_0|,1)$ contains $\alpha (d,|x_0|) r_0^{d-1}>0$ disjoint copies of $B(x_0,1)$, with 
 
\begin{eqnarray}\label{ball3}
\alpha (d, |x_0|)\equiv \alpha_0 (d) r_0^{d-1}\mbox{ as } r_0\rightarrow +\infty \mbox{ for some } \alpha_0 (d)>0.\end{eqnarray}
If we denote by $x_i $ for $i\in \{0,...,\alpha-1\}$ the centers of those balls, then we have
 
\begin{eqnarray}\label{ball2}
\int_{\bigcup\limits_{i=0}^{\alpha-1} B(x_i,1) }|u(x)|^2 dx\le \int_{\mathcal{C}(|x_0|,1)}u(r)^2 r^{d-1}dr= \omega_{d-1}\int_{r_0-1}^{r_0+1} |\tilde u(r)|^2 r^{d-1}dx \le \omega_{d-1} Br_0^{d-1},
\end{eqnarray}
\medskip

\noindent on the one hand. On the other hand, since the difference between the two crown's radius is $2$  and the balls are of radius $1$, it follows that
\begin{eqnarray}\label{ball}
|x_i|=|x_0| ,\; \forall i\in \{0,...\alpha-1\}
\end{eqnarray}
Since $u$ is radial and the balls $B(x_i,1) $ are disjoint, using (\ref{ball}) we see that
$$\int_{\bigcup\limits_{i=0}^{\alpha-1} B(x_i,1) }|u(x)|^2 dx=\alpha(d,r_0)\int_{B(x_0,1) }|u(x)|^2 dx.$$
Combining this with (\ref{ball2}) and (\ref{ball3}), we conclude the proof of item $i)$.
\medskip

$ii)$ Consider $r_0\ge1$. From geometric considerations,
there exists $\beta (d,r_0)>0$ such that the crown $\mathcal{C}(r_0,1)$ is contained in $\beta (d, r_0)$ copies of $B(0,1)$, with

\begin{eqnarray}\label{ball4}
\beta (d, r_0)\equiv \beta_0 (d) r_0^{d-1}\mbox{ as } r_0\rightarrow +\infty \mbox{ for some } \beta_0 (d)>0.\end{eqnarray}

Denoting by $y_i $ for $i\in \{0,...,\beta-1\}$ the centers of those balls, we have 
\begin{align*}\frac{1}{r_0^{d-1}}
\int_{r_0-1}^{r_0+1}|\tilde u(r)|^2 r^{d-1}dr&=\frac{1}{\omega_{d-1}r_0^{d-1}} \int_{\mathcal{C}(|x_0|,1)}|u(x)|^2 dx\\&\le\frac{1}{\omega_{d-1}r_0^{d-1}}      \sum_{i=0}^{\beta-1}  \int_{B(y_i,1) }|u(x)|^2 dx  \le  \frac{\beta (d,r_0) }{\omega_{d-1}r_0^{d-1}} A.
\end{align*}

Using (\ref{ball4}), we conclude the proof of item $ii)$.   

\end{proof}


\newpage


\begin{thebibliography}{10}

\bibitem{AS}
S.~Alexakis and A.~Shao.
\newblock On the profile of energy concentration at blow-up points for
  subconformal focusing nonlinear waves.
\newblock {\em Trans. Amer. Math. Soc.}, 369(8):5525--5542, 2017.

\bibitem{AM01}
C.~Antonini and F.~Merle.
\newblock Optimal bounds on positive blow-up solutions for a semilinear wave
  equation.
\newblock {\em Internat. Math. Res. Notices}, (21):1141--1167, 2001.

\bibitem{CZ12}
R.~C{\^o}te and H.~Zaag.
\newblock Construction of a multisoliton blowup solution to the semilinear wave
  equation in one space dimension.
\newblock {\em Comm. Pure Appl. Math.}, 66(10):1541--1581, 2013.

\bibitem{HZ2012}
M.~A. Hamza and H.~Zaag.
\newblock Lyapunov functional and blow-up results for a class of perturbations
  of semilinear wave equations in the critical case.
\newblock {\em J. Hyperbolic Differ. Equ.}, 9(2):195--221, 2012.

\bibitem{HZ12}
M.~A. Hamza and H.~Zaag.
\newblock A {L}yapunov functional and blow-up results for a class of perturbed
  semilinear wave equations.
\newblock {\em Nonlinearity}, 25(9):2759--2773, 2012.

\bibitem{HZ}
M.~A. Hamza and H.~Zaag.
\newblock Blow-up behavior for the {K}lein-{G}ordon and other perturbed
  semilinear wave equations.
\newblock {\em Bull. Sci. Math.}, 137(8):1087--1109, 2013.

\bibitem{HZ2019}
M.~A. Hamza and H.~Zaag.
 \newblock Prescribing the center of mass of a multi-soliton solution
for a perturbed semilinear wave equation. 
 \newblock {\em J. Differential Equations}, 267(6):3524--3560,
2019.

\bibitem{MZ05}
F.~Merle and H.~Zaag.
\newblock Determination of the blow-up rate for the semilinear wave equation.
\newblock {\em Amer. J. Math.}, 125(5):1147--1164, 2003.

\bibitem{MR2147056}
F.~Merle and H.~Zaag.
\newblock On growth rate near the blowup surface for semilinear wave equations.
\newblock {\em Int. Math. Res. Not.}, (19):1127--1155, 2005.

\bibitem{MR2362418}
F.~Merle and H.~Zaag.
\newblock Existence and universality of the blow-up profile for the semilinear
  wave equation in one space dimension.
\newblock {\em J. Funct. Anal.}, 253(1):43--121, 2007.

\bibitem{MR2415473}
F.~Merle and H.~Zaag.
\newblock Openness of the set of non-characteristic points and regularity of
  the blow-up curve for the 1 {D} semilinear wave equation.
\newblock {\em Comm. Math. Phys.}, 282(1):55--86, 2008.

\bibitem{MR2799813}
F.~Merle and H.~Zaag.
\newblock Blow-up behavior outside the origin for a semilinear wave equation in
  the radial case.
\newblock {\em Bull. Sci. Math.}, 135(4):353--373, 2011.

\bibitem{MR2931219}
F.~Merle and H.~Zaag.
\newblock Existence and classification of characteristic points at blow-up for
  a semilinear wave equation in one space dimension.
\newblock {\em Amer. J. Math.}, 134(3):581--648, 2012.

\bibitem{Mz12}
F.~Merle and H.~Zaag.
\newblock Isolatedness of characteristic points at blowup for a 1-dimensional
  semilinear wave equation.
\newblock {\em Duke Math. J}, 161(15):2837--2908, 2012.

\bibitem{rty}
G.~Todorova P.~Radu and B.~Yordanov.
\newblock Decay estimates for wave equations with variable coefficients.
\newblock {\em Trans. Amer. Math. Soc.}, 362(5):2279--2299, 2010.

\bibitem{GT}
V.~Georgiev~G. Todorova.
\newblock Existence of a solution of the wave equation with nonlinear damping
  and source terms.
\newblock {\em Journal of Differential Equations}, 109(2):295--308, 1994.

\end{thebibliography}

\noindent{\bf Address:}\\
\small Higher Institute for Preparatory Studies in Biology-Geology (ISEP-BG),\\
 University of Carthage, 6 Avenue 13 ao\^{u}t, 2036. La Soukra, Tunis, Tunisia
\\ \texttt{e-mail: asma.azaiez@yahoo.fr}\\

\noindent Universit\'e Sorbonne Paris Nord, Institut Galil\'ee, Laboratoire Analyse, G\'eom\'etrie et Applications,\\
CNRS UMR 7539, 99 avenue J.B. Cl\'ement, 93430 Villetaneuse, France.
\\ \texttt{e-mail: Hatem.Zaag@univ-paris13.fr}\\
\end{document}